\documentclass[11pt,a4paper]{article}
\usepackage{a4wide,amsmath,amsthm,epsfig,graphicx,psfrag,pstricks,pstricks-add}
\usepackage{amsmath,amsthm}
\usepackage{amsfonts}
\usepackage{amssymb}
\usepackage{enumerate}
\usepackage[latin1]{inputenc}  
\usepackage{bbm}

\newcommand{\R}{\mathbb{R}}
\newcommand{\N}{\mathbb{N}}

\newcommand{\F}{\mathcal{F}}
\newcommand{\llfloor}{\left\lfloor}
\newcommand{\rrfloor}{\right\rfloor}

\newcommand{\indi}[1]{\mathbbm{1}_{#1}}

\newtheorem{theorem}{Theorem} 

\newtheorem{lemma}[theorem]{Lemma}

\title{\bf A simple proof for the convexity of the Choquet integral}
\author{Aur\'elien Alfonsi\thanks{Universit\'e Paris-Est, CERMICS, Projet MathRisk
    ENPC-INRIA-UMLV, 6 et 8 avenue Blaise Pascal, 77455 Marne La Vall\'ee, Cedex
    2, France, e-mail : alfonsi@cermics.enpc.fr. This research benefited
    from the support of the ``Chaire Risques Financiers'', Fondation du
    Risque.
}}

\begin{document}
\maketitle 

{\abstract  This note presents an elementary and direct proof for the convexity of the Choquet integral when the corresponding set function is submodular.
}

{\noindent {\it Keywords: } \it Choquet integral, submodular set functions, Lov\'asz extension.\\
  {\it AMS Classification (2010): 28E10}  \\
}

Let $(\Omega, \F)$ be a measurable space and $\mathcal{X}$ be the set of functions $X: \Omega \rightarrow \R$ that are bounded and measurable with respect to the Borel $\sigma$-algebra. For $X \in \mathcal{X}$, we denote by $\|X\|=\sup_{\omega \in \Omega} |X(\omega)|$ the sup norm. We consider a set function $c: \F \rightarrow [0,1]$ that is 
\begin{itemize}
\item {\it monotone}: $\forall A, B \in \F, \ A \subset B \implies c(A) \le c(B)$, 
\item {\it normalized}: $c(\emptyset)=0$ and $c(\Omega)=1$.
\end{itemize}
The set function $c$ is said to be {\it submodular } if we have in addition 
\begin{equation} \label{def_submod}
 \forall A, B \in \F, \  c(A\cup B)+c(A\cap B) \le c(A)+c(B).
\end{equation}
For $X \in  \mathcal{X}$, the Choquet integral of $X$ with respect to $c$ is defined by
\begin{equation}
\int X dc := \int_{-\infty}^0 [c(X>x)-1]dx+ \int_0^\infty c(X>x) dx,
\end{equation}
where $c(X>x)$ is the short notation for $c(\{ \omega \in \Omega , X(\omega)>x \})$. This integral has been introduced by Choquet in its theory of capacities~\cite{Choquet} and is now used in various fields such as decision theory (see e.g. Schmeidler~\cite{Schmeidler} and Gilboa~\cite{Gilboa}) and risk measures in finance (see e.g. F\"ollmer and Schied~\cite{FS}). When $\Omega$ is finite and $\F=\mathcal{P}(\Omega)$ is the power set of~$\Omega$, the Choquet integral is also known as the Lov\'asz extension~\cite{Lovasz} and is related to optimization problems arising in machine learning or operational research, see the recent tutorial of Bach~\cite{Bach} and references within. 

It is easy to check that the Choquet integral satisfies the following properties, see Example~4.14 in~\cite{FS} :
\begin{itemize}
\item translation invariance: $\forall X \in \mathcal{X},m \in \R, \ \int (X+m) dc=\int X dc +m$,
\item monotonicity: $\forall X,Y \in \mathcal{X}, \ X\le Y \implies \int X dc \le \int Y dc$,
\item positive homogeneity: $\forall X \in \mathcal{X},\lambda>0, \ \int (\lambda X) dc= \lambda \int X dc$. 
\end{itemize}
The goal of this paper is to give an elementary proof of the following theorem.
\begin{theorem}
The set function $c$ is submodular if, and only if the Choquet integral is convex, i.e. 
$$\forall X,Y \in \mathcal{X},\forall \lambda \in [0,1], \  \int (\lambda X +(1-\lambda)Y ) dc \le  \lambda \int X dc +(1-\lambda) \int Y dc. $$
\end{theorem}
\noindent Thanks to the positive homogeneity, it is equivalent to show that   $c$ is submodular if, and only if the Choquet integral is {\it subadditive}, i.e. 
\begin{equation}\label{subadditive}
\forall X,Y \in \mathcal{X}, \  \int ( X + Y ) dc \le   \int X dc + \int Y dc. 
\end{equation}
If~\eqref{subadditive} holds, it is obvious that $c$ is submodular by taking $X=\indi{A}$ and $Y=\indi{B}$. The converse implication is instead more difficult to prove. Since the work of Choquet in~1954 (see \S54 in~\cite{Choquet}), different proofs of this result have been proposed. In Chapter 6 of~\cite{Denneberg}, Denneberg gives a proof together with a list of the different ones. Most of them require quite involved arguments, while the one of Le Cam and Buja~\cite{Buja} (see also Kindler~\cite{Kindler}) relies on a direct application of the submodular property. The proof that is proposed in the present paper also applies directly the  submodular property, but in a different way.  It can be interesting in particular for pedagogical purposes\footnote{This proof has been found while preparing the lecture on risk measures for the master students in mathematical finance of the universities Pierre et Marie Curie and Paris-Est Marne-La-Vall\'ee.}.

Let us consider then a submodular set function~$c$. Thanks to the translation invariance, we have 
 $\int ( X + Y ) dc=\int ( (X +\|X\|) + (Y+ \|Y\|))dc-\|X\|-\|Y\|$. It is therefore sufficient to prove~\eqref{subadditive} for $X\ge 0$ and $Y\ge 0$. Let $\lfloor x \rfloor$ denote the integer part of $x \in \R$. For $n \in \N^*$, we have $x-\frac{1}{n} \le \frac{\lfloor nx \rfloor}{n}\le x$, which gives by monotonicity, translation invariance and positive homogeneity
\begin{align*}
&\lim_{n\rightarrow+ \infty} \frac{1}{n} \int (\lfloor nX \rfloor )dc =\int X dc, \ \lim_{n\rightarrow+ \infty} \frac{1}{n} \int (\lfloor nY \rfloor )dc =\int Y dc,\\
& \lim_{n\rightarrow+ \infty} \frac{1}{n} \int (\lfloor nX \rfloor +\lfloor nY \rfloor  )dc =\int (X+Y) dc.
\end{align*}
Thus, it is sufficient to prove~\eqref{subadditive} for $X,Y:\Omega \rightarrow \N$. In this case, we have 
$$\int X dc =\sum_{k=1}^\infty\int_k^{k+1}c(X>x)dx=\sum_{k=1}^\infty c(X\ge k).$$
Therefore, we have $\int (X+Y) dc=\sum_{k=0}^\infty c(X+Y \ge 2k+1) + c(X+Y \ge 2k+2).$ For $k\ge 0$, we define
$$ A_k= \cup_{i=0}^{k+1} \{X \ge 2i,Y \ge 2(k-i)+1\}, \ B_k=\cup_{i=-1}^{k} \{X \ge 2i+1,Y \ge 2(k-i)\}. $$
By Lemma~\ref{lem}, we have $A_k\cup B_k=\{X+Y \ge 2k+1\} $ and  $A_k\cap B_k=\{X+Y \ge 2k+2\} $. By using the submodularity, we get
\begin{align}
&\int (X+Y) dc \nonumber\\
&\le  \sum_{k=0}^\infty c\left( \bigcup_{i=0}^{k+1} \{X \ge 2i,Y+1 \ge 2(k-i)+2\}\right) + c\left( \bigcup_{i=-1}^{k} \{X+1 \ge 2i+2,Y \ge 2(k-i)\} \right) \nonumber \\
&=\sum_{k=0}^\infty c\left( \bigcup_{i=0}^{k+1} \left\{ \llfloor \frac{ X}{2} \rrfloor \ge i,\llfloor \frac{Y+1}{2}\rrfloor  \ge k+1-i \right\}\right) + c\left( \bigcup_{i=-1}^{k} \left\{ \llfloor \frac{X+1}{2} \rrfloor \ge i+1,\llfloor \frac{Y}{2}\rrfloor \ge k-i \right\} \right) \nonumber \\
&=\sum_{k=0}^\infty c\left(\llfloor \frac{X}{2} \rrfloor+\llfloor \frac{Y+1}{2} \rrfloor \ge k+1 \right)  +c\left(\llfloor \frac{X+1}{2} \rrfloor+\llfloor \frac{Y}{2}\rrfloor \ge k+1 \right) \nonumber \\
&=\int \left(\llfloor \frac{X}{2} \rrfloor+\llfloor \frac{Y+1}{2} \rrfloor \right)dc + \int\left(\llfloor \frac{X+1}{2} \rrfloor+\llfloor \frac{Y}{2}\rrfloor\right)dc. \label{calcul_intermed}
\end{align}
We now proceed by induction and show that~\eqref{subadditive} holds when $X,Y:\Omega \rightarrow \{0,\dots,2^p\}$. For $p=0$, the submodularity gives
$\int ( X + Y ) dc=c(\{X=1\}\cup \{Y=1\})+c(\{X=1\}\cap\{Y=1\})\le c(X=1)+c(Y=1)=\int X dc + \int Y  dc$   and therefore~\eqref{subadditive} is true. Let us consider now $X,Y:\Omega \rightarrow \{0,\dots,2^{p+1}\}$. Then, $\lfloor \frac{X}{2} \rfloor$, $\lfloor \frac{Y+1}{2} \rfloor$,  $\lfloor \frac{X+1}{2} \rfloor$ and $\lfloor \frac{Y}{2}\rfloor$ take values in $\{0,\dots,2^p\}$. By using~\eqref{calcul_intermed} and the induction hypothesis, we deduce
\begin{equation*}
\int (X+Y) dc  \le \int \llfloor \frac{X}{2} \rrfloor dc+ \int \llfloor \frac{X+1}{2} \rrfloor dc +\int \llfloor \frac{Y}{2} \rrfloor dc + \int \llfloor \frac{Y+1}{2} \rrfloor dc.
\end{equation*}
It remains to observe that 
\begin{align*} \int \llfloor \frac{X}{2} \rrfloor dc+ \int \llfloor \frac{X+1}{2} \rrfloor dc&= \sum_{k=1}^\infty c\left(\frac{X}{2}\ge k\right)+c\left(\frac{X+1}{2}\ge k\right)\\&=\sum_{k=1}^\infty c(X\ge 2 k)+c(X\ge 2k-1)= \int X dc
\end{align*}
to conclude the proof. 
\vspace{0.5cm}

 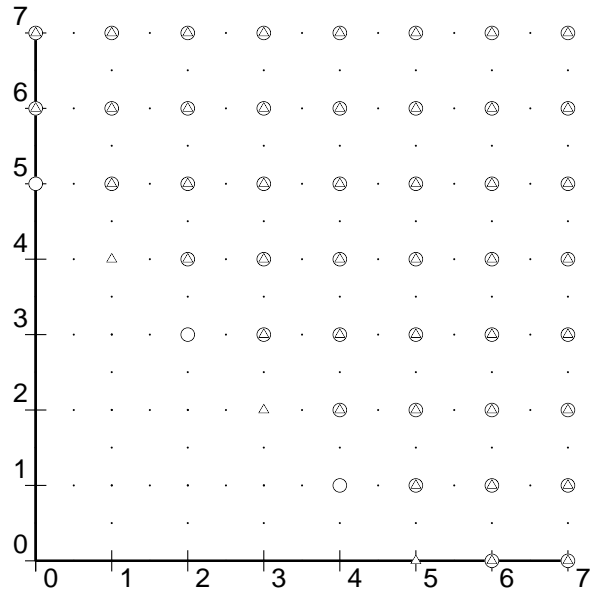
\begin{figure}[h]
    \center
    \begin{pspicture}(0,0)(7,7)      
      \psaxes[labels=none]{-}(0,0)(0,7)(7,0)      
      \psgrid[subgriddiv=1,griddots=2](0,0)(0,0)(7,7)

      \psline[linecolor=black,linestyle=none,dotstyle=o,dotscale=1.5,showpoints=true](0,7)(1,7)(2,7)(3,7)(4,7)(5,7)(6,7)(7,7)(0,6)(1,6)(2,6)(3,6)(4,6)(5,6)(6,6)(7,6)(0,5)(1,5)(2,5)(3,5)(4,5)(5,5)(6,5)(7,5)(2,4)(3,4)(4,4)(5,4)(6,4)(7,4)(2,3)(3,3)(4,3)(5,3)(6,3)(7,3)(4,2)(5,2)(6,2)(7,2)(4,1)(5,1)(6,1)(7,1)(6,0)(7,0)

      \psline[linecolor=black,linestyle=none,dotstyle=triangle,dotscale=1,showpoints=true](0,7)(1,7)(2,7)(3,7)(4,7)(5,7)(6,7)(7,7)(0,6)(1,6)(2,6)(3,6)(4,6)(5,6)(6,6)(7,6)(1,5)(2,5)(3,5)(4,5)(5,5)(6,5)(7,5)(1,4)(2,4)(3,4)(4,4)(5,4)(6,4)(7,4)(3,3)(4,3)(5,3)(6,3)(7,3)(3,2)(4,2)(5,2)(6,2)(7,2)(5,1)(6,1)(7,1)(5,0)(6,0)(7,0)

\end{pspicture}
\caption{The sets $\tilde{A}_k$ (circles) and $\tilde{B}_k$ (triangles) for $k=2$.}\label{fig1}
\end{figure}
\begin{lemma}\label{lem}
For $k \ge 0 $, we define $\tilde{A}_k=\cup_{i=0}^{k+1} \{(x,y) \in \N^2, x \ge 2i, y \ge 2(k-i)+1\}$ and $\tilde{B}_k=\cup_{i=-1}^{k} \{(x,y) \in \N^2, x \ge 2i+1,y \ge 2(k-i)\}$. We have 
$$\tilde{A}_k\cup \tilde{B}_k=\{(x,y) \in \N^2, x +y \ge 2k+1\}, \ \tilde{A}_k\cap \tilde{B}_k=\{(x,y) \in \N^2, x +y \ge 2k+2 \}. $$
\end{lemma}
\begin{proof}
This property is easier to visualize, see Figure~\ref{fig1}, but we give here a formal proof for sake of completeness. First, we notice that $x+y\ge 2k+1$ if and only if there is $i\in \{0,\dots,2k+1\}$ such that $x\ge i$ and $y\ge 2k+1-i$, which gives $\tilde{A}_k\cup \tilde{B}_k=\{(x,y) \in \N^2, x +y \ge 2k+1\}$. Similarly, if $x+y\ge 2k+2$ there is $i\in \{0,\dots,2k+2\}$ such that $x\ge i$ and $y\ge 2k+2-i$. Since $i-1\le 2 \lfloor \frac{i}{2} \rfloor \le i$ (resp.  $i-1\le 2 \lfloor \frac{i-1}{2} \rfloor +1 \le i$), this gives  $x\ge 2 \lfloor \frac{i}{2}\rfloor $ and $y\ge 2(k-\lfloor \frac{i}{2} \rfloor)+1$ (resp. $x\ge  2 \lfloor \frac{i-1}{2} \rfloor +1$ and $y\ge 2(k-\lfloor \frac{i-1}{2} \rfloor)$) and thus $(x,y) \in  \tilde{A}_k\cap \tilde{B}_k$. Conversely, if $x+y\le 2k$, we have $(x,y)\not \in \tilde{A}_k$ and $(x,y)\not \in \tilde{B}_k$ and if $x+y=2k+1$, $(x,y) \not \in \tilde{A}_k$ if $x$ is odd and $(x,y) \not \in \tilde{B}_k$ if $x$ is even.
\end{proof}

\bibliographystyle{plain}
\bibliography{bibnote}

\end{document}